\documentclass{amsart}      
\usepackage{amssymb,amsthm, amsmath, amsfonts} 
\usepackage{graphics}                 
\usepackage{hyperref}                 
\usepackage[all]{xy}
\usepackage{enumerate}

\newtheorem{theorem}{Theorem}[section]
\newtheorem{lemma}[theorem]{Lemma}
\newtheorem{proposition}[theorem]{Proposition}

\newtheorem{corollary}[theorem]{Corollary}
\newtheorem{definition}[theorem]{Definition}
\newtheorem{example}[theorem]{Example}
\newtheorem{remark}[theorem]{Remark}
\numberwithin{equation}{section}

\DeclareMathOperator{\op}{op}
\DeclareMathOperator{\pd}{pd}
\DeclareMathOperator{\Hom}{Hom}
\DeclareMathOperator{\End}{End}
\DeclareMathOperator{\Ext}{Ext}
\DeclareMathOperator{\module}{-mod}
\DeclareMathOperator{\Module}{-Mod}
\DeclareMathOperator{\add}{add}
\DeclareMathOperator{\Tria}{Tria}
\DeclareMathOperator{\tr}{tr}

\title[Triangular matrix algebras]{Derived equivalences between triangular matrix algebras}
\author{Liping Li}
\address{College of Mathematics and Computer Sciences, Hunan Normal University, Changsha, 410081, China.}
\email{lipingli@hunnu.edu.cn; lixxx480@umn.edu.}
\begin{document}

\begin{abstract}
In this paper we study derived equivalences between triangular matrix algebras using certain classical recollements. We show that special properties of these recollements actually characterize triangular matrix algebras, and describe methods to construct tilting modules and tilting complexes inducing derived equivalences between them.
\end{abstract}

\maketitle

\section{Introduction}

Throughout this paper let $A = (B, C, M)$ be a (finite dimensional) triangular matrix algebra defined by gluing two nonzero algebras $B$ and $C$ via a $(C, B)$-bimodule $M$. Our main goal is to establish various derived equivalences between triangular matrix algebras, unifying a few results described in \cite{APR, AH, Barot, Chen, Ladkani, Lin, Maycock}. By Rickard's theorem (\cite{Rickard}), two algebras are derived equivalent if and only if one algebra is isomorphic to the opposite endomorphism algebra of a certain tilting complex in the bounded derived category of the other algebra, so we focus on constructing tilting objects.

APR tilting modules were introduced in \cite{APR} for one-point extensions and one-point coextensions. Since they are special examples of triangular matrix algebras, one may wonder whether analogous tilting modules can be defined for certain triangular matrix algebras. In \cite{Li} the author gave a sufficient condition such that $T = Ae_B \oplus \tau^{-1} _A (Ae_C)$ is tilting, where $\tau_A$ is the Auslander-Reiten translation, and $e_B$ (resp., $e_C)$ is the unit in $B$ (resp., in $C$). This is called a \textit{generalized APR tilting module}. Although $T$ does induce a derived equivalence between $A$ and $\End (T) ^{\op}$, it is not guaranteed that $\End (T) ^{\op}$ is a triangular matrix algebra as well. However, we obtain an easy-to-check criteria in this paper:

\begin{theorem}
Let $A = (B, C, M)$ and suppose that $T = Ae_B \oplus \tau^{-1} _A (Ae_C)$ is a generalized APR tilting module. Then $\End_A (T) ^{\op}$ is a triangular matrix algebra glued by $B$ and $\End_A (\tau _A^{-1} (Ae_C)) ^{\op}$ if and only if $M_B$ is projective. In this case, $T$ induces a derived equivalence between $A$ and
\begin{equation*}
E = \End_A (Ae_B \oplus \tau _A^{-1} (Ae_C)) ^{\op} = \begin{bmatrix} \End_A (\tau _A^{-1} (C)) ^{\op} & 0 \\ \Hom_A (Ae_B, \tau _A^{-1} (C)) &  B\end{bmatrix}.
\end{equation*}
\end{theorem}

In \cite{HKL1} Angeleri H\"{u}gel, Koenig, and Liu described a machinery to construct tilting objects using recollements extensively investigated in \cite{HKL1, HKL2, HKLY,LVY, Koenig, P}. Since $A = (B, C, M)$, there exist certain special recollements derived from the special triangular structure. Using the strategy introduced in \cite{HKL1} and these special recollements, we obtain the following theorem unifying quite a few results in \cite{AH, Ladkani, Maycock}. \footnote{Note that different from Theorem 4.5 in \cite{LVY}, we do not need to assume that $A$, $B$, and $C$ have finite global dimensions.}

\begin{theorem}
Let $A = (B, C, M)$ and suppose that $Y \in D^b(C)$ and $Z \in D^b(B)$ are tilting objects. Then:
\begin{enumerate}
\item The object $i_{\ast} (Y) \oplus j_!(Z) \in D^b(A)$ is tilting if and only if for $n \neq 0$, $\Hom_{D^b(A)} (i_{\ast} (Y), j_!(Z)[n]) = 0$. If this holds, then $A$ is derived equivalent to
\begin{equation*}
E = \begin{bmatrix} \End_{D^b(A)} (j_!(Z)) ^{\op} & 0 \\ \Hom_{D^b(A)} (i_{\ast} (Y), j_!(Z))) & \End_{D^b(A)} (i_{\ast} (Y)) ^{\op} \end{bmatrix}.
\end{equation*}

\item Suppose that $_CM$ has finite projective dimension. Then $i_{\ast} (Y) \oplus j_{\ast} (Z) \in D^b(A)$ is a tilting object if and only if one has $\Hom_{D^b(A)} (j_{\ast} (Z), i_{\ast} (Y)[n]) = 0$ for $n \neq 0$. If this holds, $A$ is derived equivalent to
\begin{equation*}
E = \begin{bmatrix} \End_{D^b(A)} (i_{\ast} (Y)) ^{\op} & 0 \\ \Hom_{D^b(A)} (j_{\ast} (Z), i_{\ast} (Y)) & \End_{D^b(A)} ((j_{\ast} (Z)) ^{\op}) \end{bmatrix}.
\end{equation*}
\end{enumerate}
\end{theorem}

The paper is organized as follows. In Section 2 we recall the definition of generalized APR tilting modules, and prove the first theorem. Preliminary results on recollements and torsion theories of triangular matrix algebras are collected in Section 3. In the last section we use these recollements to construct tilting objects, prove the second theorem, and deduce a few applications.

We include some notation here. For an algebra $\Lambda$, $\Lambda \Module$ and $\Lambda \module$ are the category of all left $\Lambda$-modules and the category of finitely generated left modules respectively. The unbounded derived category and bounded derived category of $\Lambda \Module$ are denoted by $D(\Lambda)$ and $D^b(\Lambda)$ respectively. By $K^b (_{\Lambda} P)$ we mean the homotopy category of perfect complexes. For a $\Lambda$-module or a complex $X$, $\add (X)$ is the additive category consisting of direct summands of finite direct sums of $X$, and $\Tria (X)$ is the smallest triangulated category containing $X$. The degree shift functor $[-]$ is as usually defined.

\section{Generalized APR tilting modules}

Given two algebras $B, C$ and a $(C, B)$-bimodule $M$, the triple $(B, C, M)$ defines another algebra $A = \begin{bmatrix} B & 0 \\ M & C \end{bmatrix}$, called a \textit{triangular matrix algebra}, whose multiplication is determined by matrix product.

When $A = (B, C, M)$ satisfies the following two conditions:
\begin{itemize}
\item $C$ is a self-injective local algebra (so it is Frobenius);
\item $_CM$ has a free summand $_CC$,
\end{itemize}
we proved in \cite{Li} that $T = Ae_B \oplus \tau_A ^{-1} (Ae_C)$ is a tilting module, where $\tau_A$ is the Auslander-Reiten translation. Since this construction is a natural generalization of APR tilting modules originated in \cite{APR}, we call them \textit{generalized APR tilting modules}. In particular, if $C$ is isomorphic to the base field $k$, this generalized APR tilting module coincides with the classical APR tilting module.

Let $E = \End_A (T) ^{\op}$. The following proposition tells us under what condition $E$ is again a triangular matrix algebra.

\begin{proposition}
Let $A = (B, C, M)$ be a finite dimensional algebra such that $C$ is a self-injective local algebra and $_CM$ has a free summand $_CC$. Then $E$ is a triangular matrix algebra glued by $\End_A (\tau _A^{-1} (Ae_C)) ^{\op}$ and $B$ if and only if $M_B$ is projective.
\end{proposition}

We briefly recall the construction of $\tau _A^{-1} (Ae_C)$. For details, please refer to \cite{Li}. Since $C$ is self-injective and local, one has $DC \cong C^{\op}$, where $D = \Hom_k (-,k)$. Therefore, $DC$ has the following projective presentation as a right $A$-module: $P_1 \to e_CA \to DC \to 0$, where $P_1$ is a projective cover of $M_B$.

Applying the functor $\Hom_A (-, A_A)$ we get:
\begin{equation*}
0 \rightarrow \Hom_A (DC, A_A) \rightarrow \Hom_A (e_CA, A_A) \rightarrow \Hom_A (P_1, A_A) \rightarrow \tau _A^{-1} (C) \rightarrow 0.
\end{equation*}
Note that $\Hom_A (DC, A_A) = 0$ by the assumption that $_CM$ has a summand $_CC$ (see \cite[Lemma 4.2]{Li}), and $ \Hom_A (e_CA, A_A) \cong Ae_C = C$. Thus the above sequence turns out to be:
\begin{equation}
0 \rightarrow C \rightarrow \Hom_A (P_1, A_A) \rightarrow \tau _A^{-1} (C) \rightarrow 0,
\end{equation}
a minimal projective resolution of $\tau _A^{-1} (C)$. This sequence in general is not almost split; see \cite[Example 4.6]{Li}.

\begin{proof}
The algebra $E$ is triangular matrix if and only if $\Hom_A (\tau _A^{-1} (C), Ae_B) = 0$ since $\Hom_A (Ae_B, \tau _A^{-1} (C)) \neq 0$. Indeed, from the definition of $P_1$ and the exact sequence (2.1) one sees that the value $e_B \Hom_A (P_1, A_A)$ is not 0. Consequently, $e_B \tau _A^{-1} (C)$ and hence $\Hom_A (Ae_B, \tau _A^{-1} (C))$ are not 0.

Applying $\Hom_A (-, Ae_B)$ to (2.1) one gets
\begin{align*}
& 0 \rightarrow \Hom_A (\tau _A^{-1} (C), Ae_B) \rightarrow \Hom_A(\Hom_A (P_1, A_A), Ae_B) \rightarrow\\
& \Hom_A(C, Ae_B) \rightarrow \Ext_A (\tau _A^{-1} (C), Ae_B) \rightarrow 0.
\end{align*}
But the last term is 0 since $\tau _A^{-1} (C) \oplus Ae_B$ is a tilting module. Therefore, the first term is 0 if and only if
\begin{equation}
\Hom_A(\Hom_A (P_1, A_A), Ae_B) \cong \Hom_A(C, Ae_B) \cong e_CAe_B = M.
\end{equation}
In the left side, from the construction $P_1 \in \add (e_B A)$. Therefore, $\Hom_A (P_1, A_A) \in \add (Ae_B)$, and hence
\begin{equation}
\Hom_A(\Hom_A (P_1, A_A), Ae_B) \in \add (\End(Ae_B)) = \add (B_B).
\end{equation}

If $E$ is a triangular matrix algebra, then from (2.2) and (2.3), $_BM \in \add (B_B)$ is projective. Conversely, if $M_B$ is projective, then $M_A$ is projective as well, so $P_1 \cong M_A$. Since $\Hom_A (P_1, A_A) \in \add (Ae_B)$ and $\Hom_A (Ae_B, A) \cong \Hom_A (Ae_B, Ae_B) \cong B$, we deduce
\begin{equation*}
\Hom_A(\Hom_A (P_1, A_A), Ae_B) \cong \Hom_A(\Hom_A (P_1, A_A), A) \cong P_1 \cong M.
\end{equation*}
That is, (2.2) is true. This finishes the proof.
\end{proof}

\begin{corollary}
Let $A$ be as before and suppose that $M_B$ is projective. Then $A$ is derived equivalent to
\begin{equation*}
E = \End_A (Ae_B \oplus \tau _A^{-1} (Ae_C)) ^{\op} = \begin{bmatrix} \End_A(\tau _A^{-1} (C) ^{\op} & 0 \\ \Hom_A (Ae_B, \tau _A^{-1} (C)) &  B\end{bmatrix}.
\end{equation*}
\end{corollary}

The following example explains our construction.

\begin{example} \normalfont
Let $A$ be the path algebra of the following quiver with relations $\delta^a = \theta ^b = 0$ and $\alpha \delta = \theta \alpha$, where $a, b \geqslant 1$. Note that in this example $B = \langle 1_x, \delta, \ldots, \delta^{a-1} \rangle$ is of dimension $a$, $C = \langle 1_y, \theta, \ldots, \theta^{b-1} \rangle$ has dimension $b$, and the dimension of $M$ is $\min \{a, b \}$.
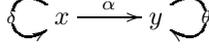
\begin{equation*}
\xymatrix{ x \ar@(ul,dl)[]|{\delta} \ar[r]^{\alpha} & y \ar@(ur,dr)[]|{\theta}}
\end{equation*}

We have three cases.

(1): If $b > a$, then $_CM$ is a proper quotient module of $_CC$, so $T = P_x \oplus \tau_A ^{-1} (P_y)$ is not a tilting module.

(2): If $a > b$, then $_CM$ is isomorphic to  $_CC$. Therefore, $T$ is a tilting module. However, since $a > b$, $M_B$ is not projective, and we conclude that the endomorphism algebra of $T$ is not triangular. To see this, let us do an explicit calculation for $a = 3$ and $b = 2$. We have:
\begin{equation*}
P_x = \begin{matrix} & & x & \\ & x & & y \\ x & & y & \end{matrix}, \quad P_y = \begin{matrix} y \\ y \end{matrix}
\end{equation*}
and the following almost split sequence:
\begin{equation*}
0 \longrightarrow P_y \longrightarrow \begin{matrix} & & x & & y \\ & x & & y & \\ x & & & & \end{matrix} \longrightarrow \tau_A^{-1} (P_y) = \begin{matrix} x \\ x \\ x \end{matrix} \longrightarrow 0.
\end{equation*}
Note that the middle term is not $P_x$. It is easy to see that both $\Hom_A (\tau ^{-1}_A (P_y), P_x)$ and $\Hom_A (P_x, \tau ^{-1}_A (P_y))$ are nonzero. Therefore, the endomorphism algebra of $T$ is not triangular.

(3): If $a = b$, then $_CM$ is isomorphic to $_CC$, and hence $T$ is tilting. Moreover, $M_B$ is isomorphic to $B_B$. So by the previous proposition, the endomorphism algebra of $T$ should be triangular. Actually, in this case the structure of $\tau_A^{-1} (P_y)$ is as shown in (2), but the socle of $P_x$ is simple and is isomorphic to $S_y$, the simple module corresponding to $y$. Therefore, $\Hom_A (\tau ^{-1}_A (P_y), P_x) = 0$, and hence the endomorphism algebra of $T$ is triangular.
\end{example}

The above result may have applications to certain finite categories related to finite groups. For example, let $k$ be a field of characteristic $p > 0$ and let $G$ be a finite $p$-group. Let $\mathcal{S}$ be a subset of subgroups of $G$. The \emph{transporter category} $\mathcal{T} _{\mathcal{S}}$ and the orbit category $\mathcal{O} _{\mathcal{S}}$ are defined as in Section 2 of \cite{Webb}. Note that the category algebras of skeletal categories of $\mathcal{T} _{\mathcal{S}}$ and $\mathcal{O} _{\mathcal{S}}$ are triangular matrix algebras. Moreover, it is easy to see that our construction of generalized APR tilting modules applies to projective modules corresponded to maximal objects in these skeletal categories.

\section{Recollements of module and derived categories}

In this section we consider certain special recollements for module categories and derived categories of triangular matrix algebras.

\begin{definition}
\cite[Definition 2.6]{PV} Let $\mathcal{C}$, $\mathcal{D}$ and $\mathcal{E}$ be abelian categories. A recollement of $\mathcal{C}$ by $\mathcal{D}$ and $\mathcal{E}$ is diagrammatically expressed as follows
\begin{equation*}
\xymatrix{\mathcal{D} \ar[rr] ^{i_{\ast}} & & \mathcal{C} \ar[rr] ^{j^{\ast}} \ar@/_1.5pc/[ll] _{i^{\ast}} \ar@/^1.5pc/[ll] _{i^!} & & \mathcal{E} \ar@/^1.5pc/[ll] _{j_{\ast}} \ar@/_1.5pc/[ll] _{j_!}}
\end{equation*}
with six additive functors $i^{\ast}, i_{\ast}, i^!, j_!, j^{\ast}, j_{\ast}$ satisfying the following conditions:
\begin{enumerate}
\item $(i^{\ast}, i_{\ast}, i^!)$ and $(j_!, j^{\ast}, j_{\ast})$ both are adjoint triples;
\item $i_{\ast}, j_!$ and $j_{\ast}$ are fully faithful;
\item the kernel of $j^{\ast}$ coincides with the image of $i_{\ast}$.
\end{enumerate}
\end{definition}

Functors appeared in the above recollement have many special properties. For instance, $i^{\ast} j_! = 0 = i^! j_{\ast}$, $i^{\ast} i_{\ast} = \text{Id}_{\mathcal{D}} = i^!i_{\ast}$, and $j^{\ast} j_! = j^{\ast} j_{\ast} = \text{Id} _{\mathcal{E}}$. Moreover, the functors $i_{\ast}$ and $j^{\ast}$ are exact.

Here is a well known example of recollements of abelian categories; see \cite{PV}.

\begin{example}
Let $\Lambda$ be a ring and let $e \in \Lambda$ be an idempotent element. Then a recollement of $\Lambda \Module$ by $e \Lambda e \Module$ and $\Lambda / \Lambda e \Lambda \Module$ is described as below:
\begin{equation*}
\xymatrix{\Lambda / \Lambda e \Lambda \Module \ar[rr] ^{\textnormal{incl}} & & \Lambda \Module \ar[rr] ^{\Hom_{\Lambda} (\Lambda e, -)} \ar@/_1.5pc/[ll] _{(\Lambda / \Lambda e\Lambda) \otimes_{\Lambda} -} \ar@/^1.5pc/[ll] ^{\Hom_{\Lambda} (\Lambda /\Lambda e\Lambda, -)} & & e\Lambda e \Module \ar@/^1.5pc/[ll] ^{\Hom _{e \Lambda e} (e\Lambda, -)} \ar@/_1.5pc/[ll] _{\Lambda e \otimes _{e\Lambda e} -}}
\end{equation*}
\end{example}

The following proposition tells us that triangular matrix algebras are characterized by special properties of these functors.

\begin{proposition}
Let $B$, $C$ and $A$ be finite dimensional algebras. Then the following are equivalent:
\begin{enumerate}
\item The algebra $A$ is isomorphic to a triangular matrix algebra $(B,C,M)$.

\item There is a recollement as follows satisfying one of the following conditions:
\begin{equation}
\xymatrix{C \Module \ar[rr] ^{i_{\ast}} & & A \Module \ar[rr] ^{j^{\ast}} \ar@/_1.5pc/[ll] _{i^{\ast}} \ar@/^1.5pc/[ll] _{i^!} & & B \Module \ar@/^1.5pc/[ll] _{j_{\ast}} \ar@/_1.5pc/[ll] _{j_!}}
\end{equation}
\begin{itemize}
\item $i^{\ast}$ preserves projective modules;
\item $i^!$ is exact.
\end{itemize}

\item There is a recollement as follows satisfying one of the following conditions:
\begin{equation}
\xymatrix{B \Module \ar[rr] ^{\iota_{\ast}} & & A \Module \ar[rr] ^{\tau^{\ast}} \ar@/_1.5pc/[ll] _{\iota^{\ast}} \ar@/^1.5pc/[ll] _{\iota^!} & & C \Module \ar@/^1.5pc/[ll] _{\tau_{\ast}} \ar@/_1.5pc/[ll] _{\tau_!}}
\end{equation}
\begin{itemize}
\item $\iota^{\ast}$ is exact;
\item $\tau_!$ is the inclusion functor.
\end{itemize}
\end{enumerate}
\end{proposition}

\begin{proof}
$(1) \Rightarrow (2)$. Let $\Lambda = A$ and $e = e_B$ in the previous example. Then one has $A/AeA \cong C$, so $i^! \cong \Hom_A (C, -)$ is exact since $C$ is a projective $A$-module. Also, since $i^{\ast} (A) \cong A/Ae_BA \otimes_A A \cong A / Ae_BA \cong C$ is projective, $i^{\ast}$ preserves projective modules.

$(1) \Rightarrow (3)$. Let $\Lambda = A$ and $e = e_C$ in the previous example. Note that $A/AeA \cong B$ is a right projective $A$-module and $\iota ^{\ast} \cong B \otimes_A -$, so $\iota^{\ast}$ is exact. Since $Ae_C \cong e_C A e_C \cong C$, one has $\tau_! \cong C \otimes_C -$, which is the inclusion functor.

$(2) \Rightarrow (1)$. A recent result of Psaroudakis and Vit\'{o}ria (Corollary 5.5 in \cite{PV}) states that any recollement of $A\Module$ is equivalent to a recollement induced by an idempotent $e$ in $A$. Therefore, without loss of generality we assume that this recollement is induced by some idempotent $e$ in $A$, so the six functors appearing in this recollement are specified as in Example 3.2.

Suppose that the functor $i^{\ast}$ preserves projective modules. Let $f = 1_A - e$. By the assumption, $i^{\ast} (Af)$ is a projective module. But
\begin{equation*}
i^{\ast} (Af) = A/AeA \otimes_A Af \cong Af / AeAAf = Af / AeAf
\end{equation*}
Note that $eAf$ is contained in the radical of $Af$, so is $AeAf$. Therefore, $i^{\ast} (Af)$ is projective if and only if $eAf = 0$, or equivalently $\Hom _A (Ae, Af) = 0$. This implies $Ae = AeA$ and $Af = fAf$. Consequently, $A$ is the triangular matrix algebra glued by $B = eAe$ and $C = fAf$ and the $(C, B)$-bimodule $fAe$, and (1) holds.

Now consider the functor $i^!$. Note that $i^! = \Hom_A (A/AeA, -)$ is exact if and only if $A/AeA$ is a projective $A$-module. But
\begin{align*}
AeA = AeAf \oplus AeAe = AeAf \oplus Ae,\\
A / AeA =  (Ae \oplus Af) / (AeAf \oplus Ae) \cong Af / AeAf.
\end{align*}
As in the previous paragraph, $A / AeA \cong Af/ AeAf$ is projective if and only if $eAf = 0$, so (1) follows.

$(3) \Rightarrow (1)$. Again, by applying Corollary 5.9 in \cite{PV} we can assume that the given recollement is induced by an idempotent $f$. Therefore, $B = A/AfA$ and $C=fAf$. Let $e = 1_A - f$. It suffices to show that each statement in (3) implies $eAf = 0$.

First, for a finitely generated $A$-module $M$,
\begin{equation*}
\iota^{\ast} (M) = (A/AfA) \otimes_A M \cong M / \tr _{Af} (M),
\end{equation*}
where $\tr _{Af} (M)$ is the trace of $Af$ in $M$. Suppose that this is an exact functor. Consider the projective module $Af$ and $\tr _{Ae} (Af) = AeAf$. Applying $\iota^{\ast}$ to the exact sequence
\begin{equation*}
\xymatrix{0 \ar[r] & AeAf \ar[r] & Af \ar[r] & Af/(AeAf) \ar[r] & 0},
\end{equation*}
we obtain an exact sequence
\begin{equation*}
\xymatrix{0 \ar[r] & \iota^{\ast} (AeAf) \ar[r] & \iota^{\ast} (Af) \ar[r] & \iota^{\ast} (Af/AeAf) \ar[r] & 0}.
\end{equation*}
Clearly, $\iota^{\ast} (Af) = 0$, so
\begin{equation*}
0 = \iota^{\ast} (AeAf) = AeAf / \tr_{Af} (AeAf).
\end{equation*}
This happens if and only if $AeAf$ is a quotient module of $(Af)^{\oplus n}$ for some $n \geqslant 0$. But $AeAf$ is generated by $eAf$, and cannot be a quotient module of $(Af)^n$ if $n > 0$. This forces $n = 0$, so $AeAf = 0$, and hence $eAf = 0$.

Now assume that $\tau_!$ is the inclusion functor, i.e., $\tau_! (N) \cong N$ as vector spaces for every $N \in C\Module$. Then we have
\begin{equation*}
\tau_! (N) = Af \otimes_{fAf} N = (fAf \otimes _{fAf} N) \oplus (eAf \otimes _{fAf} N) \cong N \oplus (eAf \otimes _{fAf} N).
\end{equation*}
Therefore, $eAf \otimes _{fAf} N = 0$ for every $N \in fAf \Module$. This happens if and only if $eAf = 0$.
\end{proof}

\begin{remark} \normalfont
The above proposition holds for categories of finitely generated modules as well.
\end{remark}

Following \cite{BR}, we define \textit{torsion pairs}.

\begin{definition}
A torsion pair of an abelian category $\mathcal{C}$ is a pair of additive full subcategories $(\mathcal{T}, \mathcal{F})$ such that:
\begin{enumerate}
\item $\Hom _{\mathcal{C}} (\mathcal{T}, \mathcal{F}) = 0$;
\item for every object $X \in \mathcal{C}$, there is a short exact sequence $0 \to T \to X \to F \to 0$ with $T \in \mathcal{T}$ and $F \in \mathcal{F}$.
\end{enumerate}
\end{definition}

If $(\mathcal{T}, \mathcal{F})$ is a torsion pair, $\Hom _{\mathcal{C}} (\mathcal{T}, F) = 0$ implies $F \in \mathcal{F}$, and $\Hom _{\mathcal{C}} (T, \mathcal{F}) = 0$ implies $T \in \mathcal{T}$.

Given $A = (B, C, M)$, the pair $(C \module, B \module)$ is a torsion pair such that both $\mathcal{T}$ and $\mathcal{F}$ are abelian categories. This property also characterizes triangular matrix algebras.

\begin{proposition}
Let $A$ be a finite dimensional algebra. Then it is isomorphic to a triangular matrix algebra if and only if $A \Module$ has a torsion pair $(\mathcal{T}, \mathcal{F})$ such that both $\mathcal{T}$ and $\mathcal{F}$ are nontrivial abelian categories.
\end{proposition}

\begin{proof}
One direction is obvious. For the other direction, suppose that both $\mathcal{T}$ and $\mathcal{F}$ are abelian categories. We claim $\Hom(\mathcal{F}, \mathcal{T}) = 0$. Indeed, for $F \in \mathcal{F}$ and $T \in \mathcal{T}$, if there is some $0 \neq \alpha: F \rightarrow T$, then $\alpha (F)$ is a submodule of $T$ as well as a quotient module of $F$. Note that $\mathcal{T}$ is closed under taking quotients and $\mathcal{F}$ is closed under taking submodules. Therefore, the kernel of $\alpha$ is contained in $\mathcal{F}$. But $\mathcal{F}$ is abelian, so $\alpha(F)$ is also contained in $\mathcal{F}$. Similarly, by considealgebra the cokernel of $\alpha$ we deduce $\alpha(F) \in \mathcal{T}$. This forces $0 \neq \alpha(F) \in \mathcal{F} \cap \mathcal{T}$, which is impossible. Actually, an even stronger result holds. That is, any $T \in \mathcal{T}$ and $F \in \mathcal{F}$ cannot have common composition factors since otherwise this common composition factor must be contained in both $\mathcal{T}$ and $\mathcal{F}$, which is impossible.

Consider the short exact sequence
\begin{equation*}
\xymatrix{0 \ar[r] & T_0 \ar[r] & A \ar[r] & F_0 \ar[r] & 0, & (\ast)}
\end{equation*}
where $T_0 \in \mathcal{T}$ and $F_0 \in \mathcal{F}$. We claim that $F_0$ is a compact projective generator of $\mathcal{F}$. Indeed, for every $F \in \mathcal{F}$, applying $\Hom_A (-, F)$ we get
\begin{equation*}
0 \rightarrow \Hom_A (F_0, F) \rightarrow \Hom_A (A, F) \rightarrow \Hom_A(T_0, F) \rightarrow \Ext_A^1 (F_0, F) \rightarrow 0.
\end{equation*}
Note that $\Hom_A(T_0, F) = 0$ implies $\Ext_A^1 (F_0, F) = 0$, so $F_0$ is projective in $\mathcal{F}$. Moreover, $\Hom_A (F_0, F) \cong \Hom_A (A, F)$. In particular, an epimorphism $A^{\oplus n} \rightarrow F$ factors through a map $(F_0) ^{\oplus n} \rightarrow F$, which must be surjective as well. Therefore, $F_0$ is a compact projective generator of $\mathcal{F}$. Let $B = \End_A (F_0)^{\op}$. By Morita theory, $\mathcal{F}$ is equivalent to $B \module$.

The fact $\Hom_A (\mathcal{F}, \mathcal{T}) = 0 = \Hom_A (\mathcal{T}, \mathcal{F})$ implies that $A$ is not contained in either $\mathcal{F}$ or $\mathcal{T}$ since otherwise either $\mathcal{T} = 0$ or $\mathcal{F} = 0$. Let $P$ be a projective cover of $F_0$ in $A \module$. We claim that $P$ is a proper summand of $A$. Indeed, if $P = A$, then the top of $F_0$ contains all simple $A$-modules (up to isomorphism). Consequently, $\mathcal{F} = 0$. This is impossible.

Choose a decomposition $A = P \oplus Q$ with $Q \neq 0$. The short exact sequence $(\ast)$ gives rise to a commutative diagram:
\begin{equation*}
\xymatrix{0 \ar[r] & K \ar[r] \ar[d] & P \ar[r] \ar[d] & F_0 \ar[r] \ar@{=}[d] & 0\\
0 \ar[r] & T_0 \ar[r] \ar[d] & A \ar[r] \ar[d] & F_0 \ar[r] & 0\\
 & Q \ar@{=}[r] & Q & &}
\end{equation*}
Therefore, $Q \in \mathcal{T}$ since it is a summand of $T_0$. Actually, $Q$ is a projective generator of $\mathcal{T}$. Indeed, for any $T \in \mathcal{T}$, $T$ and $F_0$ cannot have common composition factors. But the top of $F_0$ is isomorphic to the top of $P$. Therefore, $\Hom_A (P, T) = 0$. In other words, $T$ can be generated by $Q$. Again, by the Morita theory, $\mathcal{T}$ is equivalent to $C \module$, where $C = \End_A (Q)^{\op}$.

We finish the proof by showing $\Hom_A (P, Q) = 0$. Since $Q \in \mathcal{T}$ and $F_0 \in \mathcal{F}$, $Q$ and $F_0$ do not have common composition factors. The conclusion comes from the fact that $P$ and $F_0$ have the same top up to isomorphism.
\end{proof}

Now we turn to derived categories. \textit{Recollements} of triangulated categories are defined in \cite{CX, HKL1, HKL2}.

\begin{definition}
A recollement of a triangulated category $\mathcal{C}$ by triangulated categories $\mathcal{D}$ and $\mathcal{E}$ is expressed diagrammatically as follows
\begin{equation*}
\xymatrix{\mathcal{D} \ar[rr] ^{i_{\ast}} & & \mathcal{C} \ar[rr] ^{j^{\ast}} \ar@/_1.5pc/[ll] _{i^{\ast}} \ar@/^1.5pc/[ll] _{i^!} & & \mathcal{E} \ar@/^1.5pc/[ll] _{j_{\ast}} \ar@/_1.5pc/[ll] _{j_!}}
\end{equation*}
with six additive functors $i^{\ast}, i_{\ast}, i_!, j^!, j^{\ast}, j_{\ast}$ satisfying the following conditions:
\begin{enumerate}
\item $(i^{\ast}, i_{\ast}, i^!)$, and $(j_!, j^{\ast}, j_{\ast})$ both are adjoint triples;
\item $i_{\ast}, j_!$ and $j_{\ast}$ are fully faithful;
\item $i^! j_{\ast} = 0$;
\item for each $X \in \mathcal{C}$, there are triangles
\begin{align*}
\xymatrix{i_{\ast}i^! (X) \ar[r] & X \ar[r] & j_{\ast} j^{\ast} (X) \ar[r] & i_{\ast}i^! (X)[1]},\\
\xymatrix{j_!j^{\ast} (X) \ar[r] & X \ar[r] & i_{\ast} i^{\ast} (X) \ar[r] & j_!j^{\ast} (X)[1].}
\end{align*}
\end{enumerate}
\end{definition}

Functors appearing in the recollement have very special properties. For example, $i_{\ast}$, $j_!$, and $j_{\ast}$ are fully faithful; the composites $j^{\ast} i_{\ast}$ and $i^{\ast} j^!$ are 0; furthermore, the composites $i^{\ast} i_{\ast}$, $i^! i_{\ast}$, $j^{\ast} j_{\ast}$, and $j^{\ast} j_!$ are isomorphic to the corresponding identity functors.

The following criteria for existence of recollements of derived categories are described by H\"{u}gel, Koenig, and Liu in \cite{HKL1,HKL2,Koenig}, and by Nicol\'{a}s and Saor\'{\i}n in \cite{NS,Nicolas}.

\begin{theorem}
[\cite{Koenig, NS, Nicolas}] The derived category $D(\Lambda)$ of an algebra $\Lambda$ is a recollement of derived categories of algebras $R$ and $S$ if and only if there are objects $T_1, T_2 \in D(\Lambda)$ satisfying:
\begin{enumerate}
\item $T_1$ is compact (i.e., quasi-isomorphic to an object in $K^b(_{\Lambda} P)$), and exceptional (i.e., $\Hom _{D(\Lambda)} (T_1, T_1[n]) = 0$ for $n \neq 0$).
\item $\Hom _{D(\Lambda)} (T_2, T_2^I [n]) = 0$ for all index sets $I$ and $n \neq 0$;
\item $\Hom _{D(\Lambda)} (T_1, T_2[n]) = 0$ for all $n \in \mathbb{Z}$;
\item $T_1 \oplus T_2$ generates $D(\Lambda)$ as a triangulated category. \footnote{Here $T_1 \oplus T_2$ generates $D(\Lambda)$ if and only for $X \in D(\Lambda), \Hom_{D(\Lambda)} (T_1 \oplus T_2, X[n]) = 0$ for all $n \in \mathbb{Z}$ implies $X = 0$. But in general $D(\Lambda) \neq \Tria (T_1 \oplus T_2)$, which is the smallest triangulated category containing $T_1 \oplus T_2$.}
\end{enumerate}
Moreover, we have $\End _{D(\Lambda)} (T_1)^{\op} \cong S$ and $\End _{D(\Lambda)} (T_2)^{\op} \cong R$.
\end{theorem}

\begin{proposition}
Let $A = (B, C, M)$. Then the recollements (3.1) and (3.2) give rise to a recollement of $D(A)$.
\end{proposition}

\begin{proof}
This is a direct application of the previous theorem. For the first recollement, we let $T_1 = Ae_B$ and $T_2 = Ae_C$, both of which are compact objects in $D(A)$. For the second recollement, we can let $T_1 = Ae_C$ and $T_2 = B \cong A / (Ae_C \oplus M)$.
\end{proof}

\section{Glue tilting objects}

In this section we construct tilting objects for $A = (B, C, M)$ using the special recollements described in the previous section. Recall that an object $T \in D^b(A)$ is called a \emph{tilting object} if it is compact, exceptional, and satisfies $\Tria (T) = K^b( _AP)$. Tilting modules are special examples of tilting objects.

\begin{lemma}
Let $\Lambda$ be a finite dimensional algebra and suppose that $D^b(\Lambda)$ has the following recollement
\begin{equation}
\xymatrix{D^b(S) \ar[rr] ^{i_{\ast}} & & D^b(\Lambda) \ar[rr] ^{j^{\ast}} \ar@/_1.5pc/[ll] _{i^{\ast}} \ar@/^1.5pc/[ll] _{i^!} & & D^b(R) \ar@/^1.5pc/[ll] _{j_{\ast}} \ar@/_1.5pc/[ll] _{j_!}}.
\end{equation}
We have:
\begin{enumerate}
\item If $X \in D^b(\Lambda)$ is compact, so are $i^{\ast} (X)$ and $j^{\ast} (X)$.
\item If $j_{\ast} (Z) \in D^b (\Lambda)$ (or $j_! (Z) \in D^b(\Lambda)$) is compact, so is $Z \in D^b(R)$.
\item $Y \in D^b(S)$ is exceptional if and only if so is $i_{\ast} (Y)$.
\item $Z \in D^b(R)$ is exceptional if and only if so is $j_!(Z)$.
\item $Z \in D^b(R)$ is exceptional if and only if so is $j_{\ast} (Z)$.
\end{enumerate}
\end{lemma}

\begin{proof}
(1): By Lemma 2.3 in \cite{HKL2}, $X \in D^b (\Lambda)$ is compact if and only if for any $Y \in D^b(\Lambda)$ there is some $n_0 \geqslant 0$ (depending on $Y$) such that for all $n \geqslant n_0$, we have $\Hom _{D^b(\Lambda)} (X, Y[n]) = 0$. Note that all functors commute with the degree shift functor. When $n$ is sufficiently large, by adjunction,
\begin{equation*}
\Hom _{D^b(S)} (i^{\ast} (X), Y[n]) \cong \Hom _{D^b(\Lambda)} (X, i_{\ast} (Y[n])) = \Hom _{D^b(\Lambda)} (X, i_{\ast} (Y)[n]) = 0
\end{equation*}
for any $Y \in D^b(S)$; and
\begin{equation*}
\Hom _{D^b(R)} (j^{\ast} (X), Z[n]) \cong \Hom _{D^b(\Lambda)} (X, j_{\ast} (Z[n])) = \Hom _{D^b(\Lambda)} (X, j_{\ast} (Z)[n]) = 0
\end{equation*}
for any $Z \in D^b(R)$. Therefore, these two objects are compact.

(2): One uses the isomorphisms $j^{\ast} j_! (Z) \cong Z \cong j^{\ast} j_{\ast} (Z)$ and the conclusion that $j^{\ast}$ preserves compact objects proved in (1).

(3): For $n \neq 0$, by adjunctions one has:
\begin{equation*}
\Hom _{D^b(\Lambda)} (i_{\ast} (Y), i_{\ast} (Y)[n]) \cong \Hom _{D^b(S)} (i^{\ast} i_{\ast} (Y), Y[n]) \cong \Hom _{D^b(S)} (Y, Y[n]).
\end{equation*}
Therefore, $Y$ is exceptional if and only if the right side is 0 for $n \neq 0$; if and only if so is the left side. That is, $i_{\ast} (Y)$ is exceptional.

(4): This can be proved by the adjunction
\begin{equation*}
\Hom _{D^b(\Lambda)} (j_{\ast} (Z), j_{\ast} (Z)[n]) \cong \Hom _{D^b(R)} (j^{\ast} j_{\ast} (Z), Z[n]) \cong \Hom _{D^b(R)} (Z, Z[n]).
\end{equation*}

(5): This can be proved by the adjunction
\begin{equation*}
\Hom _{D^b(\Lambda)} (j_! (Z), j_! (Z)[n]) \cong \Hom _{D^b(R)} (Z, j^{\ast} j_!(Z)[n]) \cong \Hom _{D^b(R)} (Z, Z[n]).
\end{equation*}
\end{proof}

For a triangular matrix algebra $A = (B, C, M)$, we have an explicit description for algebras and functors in the recollement (4.1). That is,
\begin{align*}
& R = B, \quad S = C;\\
& i^{\ast} = A/Ae_BA \otimes_A^L -,\\
& i^! = R \Hom_A (C, -),\\
& j^{\ast} = R \Hom_A (Ae_B, -),\\
& j_! = Ae_B \otimes _B^L -,\\
& j_{\ast} = R\Hom_B (e_BA,-).
\end{align*}
Note that $j_{\ast}$ and $i_{\ast}$ are inclusion functors. These functors preserve extra special properties of objects, as claimed in the following lemma.

\begin{lemma}
Let $A = (B, C, M)$. Given $Z \in D^b(B)$ and $Y \in D^b(C)$, one has
\begin{enumerate}
\item $Y \in D^b(C)$ is compact if and only if so is $i_{\ast} (Y) \in D^b(A)$.
\item $Z \in D^b(B)$ is compact if and only if so is $j_! (Z) \in D^b(A)$.
\item If $Z$ is compact in $D^b(B)$ and $_CM$ has finite projective dimension, then $j_{\ast} (Z) \in D^b(A)$ is compact.
\end{enumerate}
\end{lemma}

\begin{proof}
(1): If $Y$ is compact in $D^b(C)$, then $i_{\ast} (Y)$ is compact in $D^b(A)$ since $i_{\ast}$ sends projective $C$-modules to projective $A$-modules. Conversely, if $i_{\ast} (Y) \in D^b(A)$ is compact, so is $Y \cong i^{\ast} i_{\ast} (Y)$ by (1) of the previous lemma.

(2): Note that $j_!$ also preserves projective modules, so it sends compact objects in $D^b(B)$ to compact objects in $D^b(A)$. The other direction follows from (2) in the previous lemma.

(3): If $Z$ is compact, there is a perfect complex $P^{\bullet} \in D^b(B)$ quasi-isomorphic to $Z$, so $j_{\ast} (Z)$ is quasi-isomorphic to $j_{\ast} (P^{\bullet})$. Since by our assumption $_CM$ has finite projective dimension, so does $_AB$. Therefore, each term in $j_{\ast} (P^{\bullet})$, which is contained in $\add (_AB)$, has finite projective dimension. Consequently, $j_{\ast} (P^{\bullet})$ is quasi-isomorphic to a perfect complex $Q^{\bullet}$, and hence is compact.
\end{proof}

Let $Y \in D^{b} (C)$ and $Z \in D^b(B)$ be tilting objects. The following theorem is essentially similar to Theorem 4.5 in \cite{LVY}, but we do not have to assume that $A$, $B$, and $C$ have finite global dimensions.

\begin{theorem}
Notation as above. Then $X = i_{\ast} (Y) \oplus j_!(Z) \in D^b(A)$ is a tilting object if and only if $\Hom_{D^b(A)} (i_{\ast} (Y), j_!(Z)[n]) = 0$ for $n \neq 0$. If this holds, then $A$ is derived equivalent to
\begin{equation*}
E = \begin{bmatrix} \End_{D^b(A)} (j_!(Z)) ^{\op} & 0 \\ \Hom_{D^b(A)} (i_{\ast} (Y), j_!(Z))) & \End_{D^b(A)} (i_{\ast} (Y)) ^{\op} \end{bmatrix}.
\end{equation*}
\end{theorem}

\begin{proof}
First, the compactness of $X$ is clear by the previous lemma.

Note that the functor $j_!$ maps isomorphism classes of indecomposable summands of $_BB$ bijectively to isomorphism classes of indecomposable summands of $Ae_B$. Therefore, $j_!$ gives a triangulated equivalence between $\Tria (Z) = K^b (_BP)$ and $\Tria (Ae_B)$. Thus $\Tria (j_!(Z)) = \Tria (Ae_B)$. In particular, $Ae_B \in \Tria (j_!(Z)) \subseteq \Tria(X)$. Clearly, $C \in \Tria (Y)$ since $Y$ is a tilting object, so $_AC \in \Tria (i_{\ast} (Y)) \subseteq \Tria(X)$. Consequently, $A = Ae_C \oplus Ae_B \in \Tria (X)$; that is, $\Tria (X) = K^b( _AP)$.

Now consider $\Hom_{D^b(A)} (X, X[n])$ for $n \neq 0$. Note that it can be written as
\begin{align*}
& \Hom_{D^b(A)} (i_{\ast} (Y), i_{\ast} (Y)[n]) \oplus \Hom_{D^b(A)} (j_!(Z), j_!(Z)[n]) \oplus\\
& \Hom_{D^b(A)} (i_{\ast} (Y), j_!(Z)[n]) \oplus \Hom_{D^b(A)} (j_!(Z), i_{\ast} (Y)[n]).
\end{align*}
The first two terms are zero for $i \neq 0$ by Lemma 4.1. Since $\Hom_A (Ae_B, Ae_C) = 0$, it follows that
\begin{equation*}
\Hom_{D^b(A)} (U, V) = 0, \quad \forall \, U \in \Tria (Ae_B), \forall \, V \in \Tria (Ae_C).
\end{equation*}
In particular, the last term is also 0 since $j_!(Z) \in \Tria (Ae_B)$ and $i_{\ast} (Y)[n] \in \Tria (Ae_C)$. Therefore, $X$ is exceptional if and only if $\Hom_{D^b(A)} (i_{\ast} (Y), j_!(Z)[n]) = 0$ for $n \neq 0$. This observation establishes the first statement, and the second one follows immediately.
\end{proof}

For $X = j_{\ast} (Z) \oplus i_{\ast} (Y)$, one gets a similar conclusion under an extra condition.

\begin{theorem}
Notation as above and suppose that $_CM$ has finite projective dimension. Then $X = j_{\ast} (Z) \oplus i_{\ast} (Y) \in D^b(A)$ is a tilting object if and only if for $n \neq 0$, one has $\Hom_{D^b(A)} (j_{\ast} (Z), i_{\ast} (Y)[n]) = 0$. If this holds, $A$ is derived equivalent to
\begin{equation*}
E = \begin{bmatrix} \End_{D^b(A)} (i_{\ast} (Y)) ^{\op} & 0 \\ \Hom_{D^b(A)} (j_{\ast} (Z), i_{\ast} (Y)) & \End_{D^b(A)} ((j_{\ast} (Z)) ^{\op}) \end{bmatrix}.\end{equation*}
\end{theorem}

\begin{proof}
First, the compactness of $X$ follows from Lemma 4.2. Note that here we require the projective dimension of $_CM$ to be finite to obtain the compactness of $j_{\ast} (Z)$.

Since $_CM$ and $_CC$ have finite projective dimensions, they belong to $\Tria (Y)$, so $_AM$ and $_AC$ are contained in $\Tria (i_{\ast} (Y)) \subseteq \Tria (X)$. On the other hand, $_BB \in \Tria (Z)$, so $_AB \in \Tria (j_{\ast} (Z)) \subseteq \Tria (X)$. The short exact sequence $0 \to M \to Ae_B \to B \to 0$ of $A$-modules tells us that $Ae_B \in \Tria (X)$.  Consequently, $A = Ae_B \oplus Ae_C \in \Tria (X)$; that is, $\Tria (X) = K^b( _AP)$.

One has $\Hom_{D^b(A)} (i_{\ast} (Y), i_{\ast} (Y)[n]) =0$ and $\Hom_{D^b(A)} (j_{\ast} (Z), j_{\ast} (Z)[n]) = 0$ for $n \neq 0$ by Lemma 4.1. For $n \in \mathbb{Z}$, one has by adjunction
\begin{equation*}
\Hom_{D^b(A)} (i_{\ast} (Y), j_{\ast} (Z)[n]) \cong \Hom_{D^b(A)} (Y, i^! j_{\ast} (Z)[n]) = 0
\end{equation*}
since $i^! j_{\ast} = 0$; see Definition 3.7. Therefore, $X$ is exceptional if and only if $\Hom_{D^b(A)} (j_{\ast} (Z), i_{\ast} (Y)[n]) = 0$ for $n \neq 0$. The first statement then follows from this observation, while the second statement is immediate.
\end{proof}

We apply the above theorems to a few cases which are of most interest to people.

\begin{corollary}
Let $Z \in D^{b} (B)$ be a tilting object. Then $C \oplus j_!(Z) \in D^b(A)$ is a tilting object if and only if $H^n (j_!(Z)) \in B\module$ for $n \neq 0$.
\end{corollary}

\begin{proof}
By Theorem 4.3, $C \oplus j_!(Z) \in D^b(A)$ is a tilting object if and only if $\Hom _{D^b(A)} (C, j!(Z)[n]) = 0$ for all $n \neq 0$. But
\begin{equation*}
\Hom _{D^b(A)} (C, j!(Z)[n]) \cong \Hom_A (C, H^n(j_!(Z))) \cong e_C H^n(j_!(Z))
\end{equation*}
is 0 if and only if $H^n (j_!(Z)) \in B\module$ for all $n \neq 0$.
\end{proof}

Another case is $Z = B$.

\begin{corollary}
Let $P^{\bullet} \in K^b( _C P)$ be a tilting complex. Then $X = Ae_B \oplus P^{\bullet}$ is a tilting object in $D^b(A)$ if and only if for $n \neq 0$, the map
\begin{equation*}
d_n^{\ast}: \Hom_A (P^{n+1}, Ae_B) \rightarrow \Hom_A (P^n, Ae_B)
\end{equation*}
induced by the differential map $d_n: P^n \rightarrow P^{n+1}$ is surjective.
\end{corollary}

\begin{proof}
Let $Z = B$. Then $j_! (Z) \cong Ae_B$. By Theorem 4.3, $X$ is a tilting object if and only if
\begin{equation*}
\Hom_{D^b(A)} (P^{\bullet}, Ae_B[n]) \cong \Hom_{K^b (_AP)} (P^{\bullet}, Ae_B[n]) = 0
\end{equation*}
for $n \neq 0$. That is, every chain map is a homotopy. By the diagram,
\begin{equation*}
\xymatrix{ \ldots \ar[r] & P^{n-1} \ar[r] ^{d_{n-1}} & P^n \ar[r] ^{d_n} \ar@{-->}[dl] \ar[d] ^{f_n} & P^{n+1} \ar@{-->}[dl] \ar[r] & \ldots\\
\ldots \ar[r] & 0 \ar[r] & Ae_B \ar[r] & 0 \ar[r] & \ldots}
\end{equation*}
This holds if and only if every $f_n \in \Hom_A (P^n, Ae_B)$ factors through $d_n$ for $n \neq 0$. That is, the map $d_n^{\ast}$ is surjective for $n \neq 0$.
\end{proof}

Recall a finitely generated $A$-module $T$ is a \emph{tilting module} if it satisfies the following conditions: $\pd_A T < \infty$; $\Ext_A^i (T, T) = 0$ for $i \geqslant 1$; and $\Tria (T) = K^b (_AP)$. We do not require the projective dimension of $T$ to be at most 1, as demanded by other authors in literature.

\begin{corollary}
Let $T$ be a tilting $C$-module. Then $_AT \oplus Ae_B$ is a tilting $A$-module if and only if $\Ext_A^i (T, M) = 0$ for $1 \leqslant i \leqslant \pd_C T$.
\end{corollary}

\begin{proof}
Let $Y = T$ and $Z = B$. Then $j_!(B) \cong Ae_B$ and $i_{\ast} (T)$ is $_A T$. By Theorem 4.4, $_AT \oplus Ae_B$ is tilting if and only if $\Hom _{D^b(A)} (i_{\ast} (T), j_! (B)[n]) = 0$ for $n \neq 0$; and by adjunction, if and only if $\Hom _{D^b(C)} (T, i^!j_! (B)[n]) = 0$ for $n \neq 0$. But
\begin{equation*}
i^!j_!(B) \cong \Hom_A (C, Ae_B) \cong e_CAe_B = M.
\end{equation*}
Therefore, $_AT \oplus Ae_B$ is a tilting module if and only if $\Hom _{D^b(A)} (T, M[n]) \cong \Ext_A^n (T, M) = 0$ for $n \neq 0$. The conclusion follows by observing $\pd_C T = \pd_A T$.
\end{proof}

Let $T$ be a tilting $C$-module. We consider the tilting property of $X = B[r] \oplus T[s]$. Applying degree shift if necessary, we may assume $r = 0$.

\begin{corollary}
Let $T$ be a tilting $C$-module. Then $B \oplus T[s]$ is a tilting object in $D^b(A)$ if and only if $\Ext_C^r (M, T) = 0$ for $r \neq s-1$. If this holds, $A$ is derived equivalent to
\begin{equation*}
E = \End _{D^b(A)} (B \oplus T[s])^{\op} = \begin{bmatrix} \End_C (T)^{\op} & 0 \\ \Ext_C ^{s-1} (M, T) & B \end{bmatrix}.
\end{equation*}
\end{corollary}

\begin{proof}
By Theorem 4.4, $B \oplus T[s]$ is a tilting object if and only if for $n \neq 0$, $\Hom _{D^b(A)} (B, T[n+s]) = 0$. Apply $\Hom _{D^b(A)} (-, T)$ to $M \to Ae_B \to B \to M[1]$ and note that $\Hom_{D^b(A)} (Ae_B, T[n]) = 0$ for $n \in \mathbb{Z}$. Therefore, we have
\begin{align*}
\Hom _{D^b(A)} (B, T[n+s]) & \cong \Hom _{D^b(A)} (M, T[n+s-1])\\
& \cong \Ext_A^{n+s-1} (M, T) \cong \Ext_C^{n+s-1} (M, T),
\end{align*}
and the conclusion follows.
\end{proof}

This corollary unifies some results in \cite{AH, Ladkani, Maycock}. Note that there is a little difference because in those papers the authors work on right modules and we work on left modules instead.

When $s = 1$, the above corollary tells us that $B \oplus T[1]$ is a tilting object if and only if $\Ext_A^r (M, T) = 0$ for $r \geqslant 1$. In this case, $A$ is derived equivalent to
\begin{equation*}
E = \End _{D^b(A)} (B \oplus T[1])^{\op} = \begin{bmatrix} \End_C (T)^{\op} & 0 \\ \Hom_C (M, T) & B \end{bmatrix}.
\end{equation*}
This is precisely Theorem 4.5 in \cite{Ladkani} and Theorem 5.2 in \cite{Maycock}.

Let $d = \pd_C M$. By this corollary, $B \oplus C[d+1]$ is a tilting object if and only if $\Ext_C^r (M, C) = 0$ for $0 \leqslant r \leqslant d-1$, and in this situation, $A$ is derived equivalent to
\begin{equation*}
E = \End _{D^b(A)} (B \oplus C[d+1])^{\op} = \begin{bmatrix} C & 0 \\ \Ext_C^d (M, C) & B \end{bmatrix}.
\end{equation*}
This is Theorem 2.1 in \cite{AH}.

The reader may want to decompose tilting objects in $D^b(A)$ to tilting objects in $D^b(B)$ and $D^b(C)$ using functors $j^{\ast}$, $i^{\ast}$ and $i^!$. Unfortunately, this is not the case in general. For example, if we apply the functor $i^{\ast} = C \otimes _A^L -$ to the generalized APR tilting module $T$ we will get 0, since $T$ is generated by $e_BT$. It is also easy to see that $i^!$ does not preserve exceptional property. However, the functor $j^{\ast}$ preserves tilting modules.

\begin{lemma}
Let $\Lambda$ be a finite dimensional algebra and suppose that $D(\Lambda)$ has a recollement (3.3). Let $T \in D^b(\Lambda)$ be a tilting object. Then $j^{\ast} (T)$ is a tilting object if and only if for $n \neq 0, 1$, $\Hom_{D^b(\Lambda)} (T, i_{\ast}i^! (T) [n]) = 0$, and there is an exact sequence
\begin{align*}
& 0 \rightarrow \Hom_{D^b(\Lambda)} (T, i_{\ast}i^! (T)) \rightarrow \End_{D^b(\Lambda)} (T) \rightarrow \Hom_{D^b(\Lambda)} (T, j_{\ast}j^{\ast}(T))\\
& \rightarrow \Hom_{D^b(\Lambda)} (T, i_{\ast}i^! (T)[1]) \rightarrow 0.
\end{align*}
\end{lemma}

\begin{proof}
We know that $j^{\ast} (T)$ is compact. Since $_{\Lambda} \Lambda \in \Tria (T)$, $_RR = j^{\ast} (\Lambda) \in \Tria (j^{\ast} (T))$. It suffices to show the exceptional property of $j^{\ast} (T)$. This is equivalent to the exceptional property of $j_{\ast} j^{\ast} (T)$.

Applying $\Hom _{D^b(\Lambda)} (-, j_{\ast}j^{\ast}(T))$ to
\begin{equation*}
i_{\ast} i^! (T) \rightarrow T \rightarrow j_{\ast} j^{\ast} (T) \rightarrow i_{\ast} i^!(T) [1]
\end{equation*}
and using $\Hom_{D^b(\Lambda)} (i_{\ast} i^! (T), j_{\ast} j^{\ast} (T)[n]) = 0$ for $n \in \mathbb{Z}$, we conclude
\begin{equation*}
\Hom _{D^b(\Lambda)} (T, j_{\ast}j^{\ast} (T) [n]) \cong \Hom _{D^b(\Lambda)} (j_{\ast}j^{\ast} (T), j_{\ast}j^{\ast} (T) [n])
\end{equation*}
for $n \in \mathbb{Z}$. Applying $\Hom _{D^b(\Lambda)} (T, -)$ to the same triangle, we conclude that $\Hom _{D^b(\Lambda)} (T, j_{\ast}j^{\ast} (T) [n]) = 0$ for $n \neq 0$ if and only if $\Hom_{D^b(\Lambda)} (T, i_{\ast}i^! (T) [n]) = 0$ for $n \neq 0, 1$, and there is an exact sequence as specified.
\end{proof}

\begin{proposition}
Let $A = (B, C, M)$ and let $T \in A\module$ be a tilting module with projective dimension at most 1. Then $j^{\ast} (T)$ is a tilting $B$-module with projective dimension at most 1.
\end{proposition}

\begin{proof}
Note that $j^{\ast} = \Hom_A (Ae_B, -)$ is exact and preserves projective modules. Thus applying $j^{\ast}$ to a minimal projective resolution of $T$ and a minimal $T$-resolution of $_AA$, we deduce that $\pd_B j^{\ast} (T) = \pd_B e_BT \leqslant \pd_A T \leqslant 1$, and $_BB$ has a $j^{\ast}(T)$-resolution. It remains to show $\Ext_B^1 (e_BT, e_BT) = 0$.

Since $\pd_A T \leqslant 1$ and $i_{\ast}i^! (T) = e_CT$, $\Hom_{D^b(A)} (T, i_{\ast}i^! (T) [n]) = 0$ for $n \neq 0, 1$. Moreover, applying $\Hom_A(T, -)$ to the short exact sequence $0 \rightarrow e_C T \rightarrow T \rightarrow e_BT \rightarrow 0$ one obtains the exact sequence
\begin{equation*}
0 \rightarrow \Hom_A (T, e_CT) \rightarrow \End_A (T) \rightarrow \Hom_A (T, e_BT) \rightarrow \Ext_A^1 (T, e_CT) \rightarrow 0.
\end{equation*}
Thus by the previous lemma, $e_BT$ is a tilting object in $D^b(B)$.
\end{proof}

\end{document}